\documentclass[3p,times,12pt]{article}
\usepackage{graphicx}
\usepackage{longtable}
\usepackage{amsmath,amssymb,amsthm,amsfonts}
\usepackage{color}
\usepackage{blkarray}
\usepackage{dot2texi}

\usepackage{tikz}
\usepackage{tkz-graph}
\usepackage{float}

\newtheorem{theorem}{Theorem}[section]
\newtheorem{prop}[theorem]{Proposition}
\newtheorem{lemma}[theorem]{Lemma}
\newtheorem{exam}[theorem]{Example}
\newtheorem{conj}[theorem]{Conjecture}

\newtheorem{cor}[theorem]{Corollary}
\newtheorem{obs}[theorem]{Observation}

\newtheorem{question}[theorem]{Question}
\newenvironment{pf}{\medskip {Proof:  \hspace*{-.4cm}}
    \enspace}{\hfill \qed \newline \medskip}

\DeclareMathOperator{\CCR}{CCR}

\newcommand{\sqr}{\: \Box\:  }

\DeclareMathOperator{\tw}{tw}

\DeclareMathOperator{\rank}{rank}

\title{The $q$-Analogue of Zero Forcing for Certain Families of Graphs}
\author{Shaun Fallat \footnotemark[1] \footnotemark[2]  \footnotemark[8]
\and
Neha Joshi \footnotemark[1]  \footnotemark[2]
\and
Roghayeh Maleki \footnotemark[7]
\and
Karen Meagher \footnotemark[1] \footnotemark[2]
\and
Seyed Ahmad Mojallal \footnotemark[1] \footnotemark[2]
\and
Shahla Nasserasr \footnotemark[3]
\and
Mahsa N. Shirazi \footnotemark[1] \footnotemark[5]
\and
Andriaherimanana Sarobidy Razafimahatratra \footnotemark[7]
\and
Brett Stevens \footnotemark[4]
}
\date{\today}

\begin{document}

\maketitle

\renewcommand{\thefootnote}{\fnsymbol{footnote}}

\footnotetext[1]{Department of Mathematics and Statistics,
University of Regina, Regina, SK, S4S 0A2. }
\footnotetext[2]{Fallat: \texttt{shaun.fallat@uregina.ca}; Joshi: \texttt{njp008@uregina.ca}; Meagher: \texttt{Karen.Meagher@uregina.ca}; Mollojal: \texttt{ahmad\_mojalal@yahoo.com}.}
\footnotetext[3]{School of Mathematical Sciences, Rochester Institute of Technology, Rochester, NY. 14623. \texttt{sxnsma@rit.edu}}
\footnotetext[4]{School of the Mathematics and Statistics, Carleton University, Ottawa, ON,  K1S 5B6. \texttt{brett@math.carleton.ca}}
\footnotetext[5]{Department of Mathematics, University of Manitoba, Winnipeg, MB, R3T 2N2. \texttt{mahsa.nasrollahi@gmail.com}}
\footnotetext[6]{University of Primorska, UP FAMNIT, Glagolja\v ska 8, 6000 Koper, Slovenia. \texttt{r.maleki910@gmail.com}}
\footnotetext[7]{University of Primorska, UP FAMNIT, Glagolja\v ska 8, 6000 Koper, Slovenia. \texttt{sarobidy@phystech.edu}}
\footnotetext[8]{Corresponding author.}
\renewcommand{\thefootnote}{\arabic{footnote}}

\begin{abstract}
Zero forcing is a combinatorial game played on a graph with the ultimate goal of changing the colour of all the vertices at minimal cost.  Originally this game was conceived as a one player game, but later a two-player version was devised in-conjunction with studies on the inertia of a graph, and has become known as the $q$-analogue of zero forcing. In this paper, we study and compute the $q$-analogue zero forcing number for various families of graphs. We begin with by considering a concept of contraction associated with trees. We then significantly generalize an equation between this $q$-analogue of zero forcing and a corresponding nullity parameter for all threshold graphs. We close by studying the $q$-analogue of zero forcing for certain Kneser graphs, and a variety of cartesian products of structured graphs.

\vspace{.4cm}

\vspace{.4cm}
\end{abstract}

\section{Introduction}

Zero forcing, originally conceived as method for bounding the maximum nullity of a graph, is a combinatorial game played on a graph.  This propagation-type game involves colouring the vertices of a graph by applying a sequence of two possible moves or operations on the vertices. The key move is known as the  \emph{colour change rule}. We note that sometimes this rule is also referred to as the forcing rule or the filling rule. The colour change rule (or CCR for short) is defined as: if a coloured vertex has a unique uncoloured neighbour (and any number of coloured neighbours), then this exactly one uncoloured neighbour becomes coloured at no cost.  In the literature, and sometimes here as well, when a coloured vertex changes the colour of unique uncoloured vertex, this is often referred to as a coloured vertex {\em forces} an uncoloured vertex. The game is summarized as follows.

\begin{quote}
{\bf The Zero Forcing Game (or $Z$-game):} All the vertices of the graph $G=(V,E)$ are initially uncoloured and there is one player who has tokens.  The player will repeatedly apply one of the following two operations until all vertices are coloured:
\begin{enumerate}
\item For one token, any vertex can be changed from uncoloured to coloured.
\item At no cost, the player can apply the CCR on $G$.
\end{enumerate}
\end{quote}

The {\em zero forcing number} for the graph $G$, denoted by $Z(G)$, is the minimum number of tokens needed to guarantee that all vertices can be coloured.  Zero forcing was originally developed in the combinatorial matrix theory community to provide a combinatorial bound for the minimum rank (or more specifically the maximum nullity) of a symmetric matrix $A$ associated with a graph \cite{AIM}.  In fact, the CCR above describes when a collection of zero coordinates in a given null vector for  such a matrix $A$, associated to a graph $G$, implies that this null vector is equal to the zero vector. 

Since the work in \cite{AIM} appeared, there have been many variations on the $Z$-game, including positive semidefinite $Z$-game, along with many others corresponding to particular subclasses of matrices that can be associated with a graph (or perhaps a directed graph). Most relevant to this work and the one considered here is a $q$-analogue of zero forcing  which introduces a third operation available to the player (see \cite{BGH}).  This new operation allows the (potential) application of the CCR on an induced subgraph of the original graph, and depends on the choice made by a second player, often referred to as the {\em oracle}.  For any graph $G$, if $X\subseteq V(G)$, then we let $G[X]$ be the induced subgraph of $G$ on the vertices $X$.
\begin{quote}
{\bf The $q$-Analogue of the Zero Forcing Game (or $Z_q$-game):}  For $q\geq 0$ an integer, assume all the vertices of the graph $G=(V,E)$ are initially uncoloured and there is one player who has tokens, and one oracle.  The player will repeatedly apply one of the following three operations until all vertices are coloured.
\begin{enumerate}
\item For one token, any vertex can be changed from uncoloured to coloured.
\item At no cost, the player can apply the CCR on $G$.
\item Let the vertices currently coloured be denoted by $B$, and $W_1,\ldots,W_k$ be the vertex sets of the connected components of $G[V\setminus B]$ (i.e., components of uncoloured vertices).  If $k\ge q+1$, the player selects \emph{at least}\/ $q+1$ of the $W_i$'s and announces the selection to the oracle.  The oracle selects a nonempty subset of these components, $\{W_{i_1},\ldots,W_{i_\ell}\}$, and announces it back to the player.  At no cost, the player can apply the CCR on $G[B\cup W_{i_1}\cup\cdots\cup W_{i_\ell}]$. 
\end{enumerate}
\end{quote}

The {\em $q$-analogue zero forcing number} of a graph, denoted by $Z_q(G)$, is the minimum number of tokens needed to guarantee that all vertices can be coloured, regardless of how the oracle responds.  
For convenience, we may refer to the third option above as the CCR$_q$ for a fixed $q$.

Due to the ``at least'' part in the definition above it follows that  $Z_0(G)\le Z_1(G)\le\cdots\le Z(G)$. We note here that the positive semidefinite zero forcing number (see \cite{peters}), and is usually denoted by $Z_{+}(G)$, is equal to the number $Z_{0}(G)$ as defined above. Further, in \cite{BGH} it was shown that the parameter $Z_q(G)$ also provides an upper-bound related to the maximum nullity over a certain class of matrices associated with $G$.  

\begin{theorem}[\cite{BGH}]
If $A$ is a real-symmetric matrix with nonzero off-diagonal entries corresponding to the edges of a graph $G$ \emph{and} $A$ has at most $q$ negative eigenvalues, then the nullity of $A$ is at most  $Z_q(G)$.
\end{theorem}

It is useful to note that for the $Z$-game it is well-known that the spending of tokens can all happen up front before applying the CCR.  Hence, in previous works there is a focus on zero forcing \emph{sets} for the $Z$-game.  In the $Z_q$-game it might be disadvantageous to spend all tokens up front, i.e., the oracle's response(s) may change the optimal spending pattern (see an example given in \cite{BGH}). Consequently, computing $Z_q(G)$ is rather tricky, partly due to the vast array of options at each stage of the game. For example, if we restrict our attention to trees, the maximum nullity of all such symmetric matrices associated to a tree with exactly $q$ negative eigenvalues (denoted by $M_q$) is known and depends on the so-called maximum disconnection number of a tree (see \cite{BHL}). However, beyond some very basic trees (paths, stars, etc.) $Z_q$ for a tree is very complicated to compute (cf. Section 2.1) for $0 < q < n$. It is our purpose here to derive $Z_q$ for a number of families of graphs in an effort to develop more data for both the zero forcing numbers of these graphs and for bounding the corresponding nullities. Such computations will provide some insights to the very important inverse eigenvalue problem for graphs.

In this paper, we begin by deriving a tool, we call a contraction intended to better illuminate the CCR$_{q}$ for the case of trees. 
In Section 3, we consider the class of connected threshold graphs and establish that $M_q=Z_q$ for every $q$ and we derive a formula for this number as well. Finally, in Section 4, we consider a variety of specific families of graphs and compute the $q$-analogue zero forcing number for such graphs. In Section 4, we also state numerous conjectures where there are some gaps in our computations.

\section{A Contraction on Graphs}
Let $G$ be a graph with vertex set $V$. Assume further, that $B$ denotes the subset of coloured vertices and let $W=V\setminus B$.  
We define $C(G)$, the {\em bipartite contraction of $G$}  by contracting edges incident to two coloured vertices. All edges joining two uncoloured vertices are also contracted.  These contractions are repeated until all uncoloured vertices are only connected to coloured vertices and each coloured vertex is connected to at least one uncoloured vertex.  
For such a graph $G$, we use the term {\em coloured} (or {\em uncoloured}) {\em component} of $G$ to mean connected induced subgraphs consisting of only coloured (or uncoloured vertices). Given a graph $G$ with vertices $B$ coloured and $W=V\setminus B$, and $S \subset W$, let $G\{S\}$ be the graph obtained from $G$ that is induced by $S$ along with all coloured components (namely connected induced subgraphs consisting of only coloured vertices) with an edge connected to $S$.


\begin{lemma}\label{lemma_bipart}
Let $G$ be a graph with vertices $B$ coloured and $W=V\setminus B$ and let $C(G)$ be the bipartite contraction of $G$.  If there exists a set of uncoloured vertices $W$ of $C(G)$ such that $C(G)\{S\}$ has a coloured vertex of degree 1 for all $\emptyset \neq S \subseteq W$, then in the game $\CCR_{i}$ for $G$, there is a move which will permit at least one new vertex to be forced  for any $i \leq |W|-1$.
\end{lemma}
\begin{proof}
Consider the uncoloured components in $G$ corresponding to $W$ and hand them to the oracle. If the oracle passes back a set $S$ of components, then in $C(G)\{S\}$  there is a coloured vertex $x$ of degree 1. In the subgraph of $G$ induced on the coloured vertices and the uncoloured components corresponding to $S$, there exists a vertex in the coloured component corresponding to $x$ which has exactly one uncoloured neighbour  and only one edge joining it to this neighbour.  This uncoloured neighbour can be forced.
\end{proof}

\begin{theorem}\label{thm_sufficient}
  Let $G$ be a tree with vertices $B$ coloured and $W=V\setminus B$ and suppose $C(G)$ be the bipartite contraction of $G$. If there is a matching of size $q+1$ in $C(G)$, then there is a move which will permit at least one new vertex in $G$ to be forced using the  $\CCR_{q}$.
  \end{theorem}
  \begin{proof}
   Assume $G$ is a tree with $B$ and $W$ given as assumed. Then $C(G)$ does not have a cycle. Suppose $M = \{\{w_i,b_i\}, i=1,\ldots,q+1\}$ is a matching of size $q+1$ in $C(G)$ and let $S = \{w_{i_j}, j=1,\ldots,s\} \subset W$.  Assume that the coloured vertices of $C(G)\{S\}$ are given by $\{b_{i_j}, j=1, \ldots, t\}$ with $t \geq s$ because there is a matching saturating $S$.  Since $C(G)\{S\}$ is acyclic, it contains $e \leq  s+t-1$ edges.  Note that every vertex of $B$ has at least one white neighbour. If all coloured vertices have degree at least 2 in $C(G)$, then $e \geq 2t$.  This implies that $t < s$, which is a contradiction.  
Thus there exist at least one coloured vertex of degree one for all $\emptyset \neq S \subseteq W$. Using Lemma \ref{lemma_bipart}, a new vertex can be forced if the oracle returns $S$.
    \end{proof}

For the next result we incorporate Hall's theorem as a tool to aid the proof, and as such if $X \subseteq V$, we let $Nbd(X)$ denote the set of vertices adjacent to at least one vertex in $X$.

\begin{theorem}\label{thm_necessary}
  Let $G$ be a connected graph with vertices $B$ coloured and $W=V\setminus B$, and let $C(G)$ be the bipartite contraction of $G$. If there is a move which will permit at least one new vertex in $C(G)$ to be forced in $\CCR_{q}$ (and not $\CCR$) then there exists a matching of size $q+1$ in $C(G)$.
\end{theorem}
\begin{proof}
 Suppose that the set of at least $q+1$ uncoloured vertices $W$ in $C(G)$ represents a move in game $\CCR_q$ which allows at least one new vertex to be forced. We show that the graph $C(G)\{W\}$ satisfies Hall's Theorem and thus contains a matching of size $|W|$ saturating the vertices of $W$.  First, since $G$ is connected, every singleton of $W$ has at least one coloured neighbour.  Suppose  $S \subset W$ is the minimal counter example to Hall's theorem, namely $|Nbd(S)| < |S|.$
    If the oracle returns any of the uncoloured components corresponding to $S$, at least one uncoloured vertex can be forced so there must exist a coloured vertex $b \in Nbd(S)$ with exactly one uncoloured neighbour, $w$.  Let $S' = S \setminus \{w\}$ in $C(G)\{S\}$.  Using the minimality of $S$ we have

    \[
      |S'| \leq |Nbd(S')| \leq |Nbd(S)| -1 < |S| -1 = |S'|.
    \]
    Thus for all $S \subset W$,
   $ |Nbd(S)| \geq |S|$
  and there exists a matching saturating $S$ which must have size at least $q+1$.
\end{proof}

%

\section{$Z_{q}(G)$ and $M_q(G)$ of a threshold graph}

Given a graph $G$ on $n$ vertices, we let $S(G)$ denote the collection of all real $n \times n$ symmetric matrices whose off-diagonal entry in position $(i,j)$ is assigned a nonzero number if and only if vertices $i$ and $j$ are adjacent. For this class of symmetric matrices we let $M(G)$ denote the maximum nullity over this class.  On the other hand, given $q$, we may restrict attention to those symmetric matrices governed by the edges of $G$ with exactly $q$ negative eigenvalues (counting multiplicities) and label this set as $S_q(G)$. For such matrices we consider their largest possible nullity and denote this maximum by $M_q (G)$. As noted in the introduction we have $M_q (G) \leq Z_q (G)$ for all $q$.

All threshold graphs are obtained through an iterative process which starts with an isolated vertex, and where, at
each step, either a new isolated vertex is added, or a vertex adjacent to all previous vertices (or dominating
vertex) is added. A vertex in a threshold graph that is adjacent to every other vertex in the graph is called a {\em universal} vertex.

We may represent a threshold graph on $n$ vertices using a binary sequence $(b_1, \ldots, b_n)$. Here $b_i$ is 0 if vertex $v_i$ was added as an isolated vertex, and $b_i$ is 1 if $v_i$ was added as a dominating vertex. This representation has been called a {\em creation sequence} (see \cite{HSS}). For brevity, if $G$ is a threshold graph with creation sequence $(b_1, \ldots, b_n)$ we write $G \cong (b_1, \ldots, b_n)$. We assume 0 is the first character of the string; it represents the first vertex of the graph. This way, the number of characters equal to 1 in this string, called the {\em trace} of the graph, indicates the number of dominating vertices in its construction \cite{Mer}. As we are interested in connected threshold graphs, we always assume that $b_n =1$. 

 For a threshold graph $G$, it is well-known that 
 \begin{equation}\label{well-known on Z}
 Z(G)=M(G)=n-2T+s_1+2s_0,
 \end{equation} where $s_1$ is the number of  $01$ patterns  (only happens for the first two entries of the creation sequence) or $101$ and $s_0$ is the number of $11$  patterns in the creation sequence of $G$ (see \cite{ShuWan}). 
 
 Also for any chordal graph $G$ on $n$ vertices, it is well-known that $Z_0(G)=M_0(G)=n-cc(G)$ (see \cite{barioli2013parameters}), where $cc(G)$ is known as the clique cover number of $G$ and is equal to the fewest number of cliques needed to cover the edges of $G$. We know that a threshold graph is a chordal graph and also each clique in a minimum clique cover includes an isolated vertex. Indeed, for a threshold graph $G$,  $cc(G)$ is the number of zeros in the creation sequence which equals $n-T$. Hence for a threshold graph $G$, we have 
 \begin{equation}Z_0(G)=M_0(G)=T. \label{mo} \end{equation}
 
 In this section, we generalize equations (\ref{well-known on Z}) and (\ref{mo}) and establish $M_q(G)=Z_q(G)$ for all connected threshold graphs $G$ and for all $q$. In addition, we derive the following formula for this common value, which is provided below.
Suppose $G$ is a connected threshold graph with creation sequence given by 
\[
(0^{(k_1)}, 1^{(t_1)}, \ldots,0^{(k_s)}, 1^{(t_s)}):=
(\underbrace{0, \ldots, 0}_{k_1},\underbrace{1, \ldots, 1}_{t_1}, \ldots,  \underbrace{0, \ldots, 0}_{k_s}, \underbrace{1, \ldots, 1}_{t_s}),\]
 with trace $T=\sum_{i=1}^s t_i$ and let $a_j=\max\{k_j-2,\,0\}$ for $j\in \{1,2,\ldots,s\}$. Then for any 
$q\in \{0,1,2,\ldots,s\}$, 
 \begin{equation}\label{z_q}
M_q(G)= Z_q(G)=T+\max_{1 \le i_1<\ldots<i_q \le s} \sum_{j=1}^q a_{i_j}.
 \end{equation}

We begin by verifying that the right-hand side of equality in (\ref{z_q}) is valid as an upper bound for $Z_q$ for any connected  threshold graph. 

\begin{theorem}\label{thm:upperbound on z_q}
 Let $G\cong (0^{(k_1)}, 1^{(t_1)}, \ldots,0^{(k_s)}, 1^{(t_s)})$ be a connected threshold graph with trace $T=\sum_{i=1}^s t_i$ and let $a_j=\max\{k_j-2,\,0\}$ for $j\in \{1,2,\ldots,s\}$. Then for any $q\in \{0,1,2,\ldots,s\}$, 
 \begin{equation}\label{upper on z_q}
 Z_q(G)\le T+\max_{1 \le i_1<\ldots<i_q \le s} \sum_{j=1}^q a_{i_j}.
 \end{equation}
\end{theorem}
\begin{pf}
As stated above, we already know that $Z_0(G)=T$, so we assume $q\geq 1$. To start the $Z_q$-game, we initially assign a total of $T=\sum_{i=1}^s t_i$ tokens as follows: to each group of dominating vertices of size $t_i$, assign $t_i-1$ tokens arbitrarily and to each group of isolated vertices we assign one token arbitrarily. 
Now, applying the CCR we can force the remaining uncoloured dominating vertex in each group of dominating vertices, by starting from the right-most group (namely the $s$th group) in the creation sequence and move leftwards. Suppose that $G_1$ is the induced subgraph obtained from removing all coloured vertices, that is, removing all dominating vertices together with one isolated vertex located in each isolated group. It is clear that by applying the  $\CCR_q$, the oracle will try to force all isolated vertices until $q$ largest isolated groups, called $I_{k_{i_1}-1},\, I_{k_{i_2}-1},\, \ldots,\, I_{k_{i_{q}-1}}$ remain. 

Now if $k_{i_j}\ge 3$, then we consider spending $k_{i_j}-2$ tokens in the corresponding group $I_{i_j}$, otherwise we do not spend any tokens in groups $I_{i_j}$ when $k_{i_j}<3$. In this way, all remaining vertices can be coloured by applying the CCR from the left-most group (or smallest indexed group) in the creation sequence and moving to the right. Hence we have spent an additional 
$$\sum_{j=1}^p (|I_{k_{i_j}-1}|-1)=\sum_{j=1}^p (k_{i_j}-2)=\sum_{j=1}^p a_{i_j}=\sum_{j=1}^q a_{i_j},$$ 
where $p$ is the number of indices $i_j$ for which  $k_{i_j}\ge 3$, 
tokens to complete  $Z_q$-game. Hence,  the total number of tokens spent in this process is given by $$T+\max_{1 \le i_1<\ldots<i_q \le s} \sum_{j=1}^q a_{i_j}.$$
\end{pf}

Having established the desired upper bound, we now turn our attention to the lower bound corresponding to equation (\ref{z_q}). We state this lemma now with an additional condition on a diagonal entry corresponding to a universal vertex in the connected threshold graph $G$. For simplicity if $B$ is an $n \times n$ matrix, we let $\eta(B)$ denote the nullity of $B$ and $B_{i,j}$ denotes the $(i,j)$ entry of $B$. Further, if $w \in \{1,2,\ldots, n\}$, then $B(w)$ denotes the principal submatrix of $B$ obtained by deleting row and column $w$ from $B$.

\begin{lemma}\label{thm:lowerbound on z_q}
 Let $G\cong (0^{(k_1)}, 1^{(t_1)}, \ldots,0^{(k_s)}, 1^{(t_s)})$ with the trace $T=\sum_{i=1}^s t_i$ and let $a_j=\max\{k_j-2,\,0\}$ for $j\in \{1,2,\ldots,s\}$. Then for any $1 \leq q \leq s$, there exists a matrix $B \in S_q(G)$ such that 
 \begin{equation}\label{lower on z_q}
M_q(G) \geq \eta(B) \geq  T+\max_{1 \le i_1<\ldots<i_q \le s} \sum_{j=1}^q a_{i_j},
 \end{equation}
 and for at least one universal vertex $w$ in $G$, $B_{w,w}\neq 0$.
\end{lemma}

We will defer the formal inductive proof of the above lemma as its verification will involve another result stated below (see Lemma \ref{ts=1}). However, we do note that Lemma \ref{thm:lowerbound on z_q} does hold for $q=0$, since we already know that  for any threshold graph $G$, $Z_0(G)=M_0(G)=T$, and positive semidefinite matrices must have positive main diagonal entries corresponding to nonzero rows.

Before we proceed, we need an auxiliary result requiring that Lemma \ref{thm:lowerbound on z_q} hold for certain values of $q$ by induction. This result is  concerned with the nullity of matrices associated with a particular family of connected threshold graphs 

\begin{lemma}
Let $G\cong (0^{(k_1)}, 1^{(t_1)}, \ldots,0^{(k_s)},1^{(1)})$ be a connected threshold graph with $n$ vertices and trace $T$. Assume that Lemma \ref{thm:lowerbound on z_q} holds for all $q$ such that $q<q_0$, where $q_0\geq 1$. If $a_s$ is among the $q_0$ maximum elements from the collection $\{a_1, a_2, \ldots, a_s\}$, then there exists a matrix $B \in S_{q_0}(G)$ such that the nullity of $B$ is at least 
\[T+\max_{1 \le i_1<\ldots<i_{q_{0}-1} \le s-1} \sum_{j=1}^{q_{0}-1} a_{i_j} + a_s,\] and if $w$ is the universal vertex of $G$, then we have $B_{w,w}\neq 0$.
 \label{ts=1}
 \end{lemma}
 
 \begin{pf}
 Let $H\cong (0^{(k_1)}, 1^{(t_1)}, \ldots,0^{(k_{s-1})}, 1^{(t_{s-1})})$. Then, by the assumption applied to Lemma \ref{thm:lowerbound on z_q},  $\exists B_1\in S_{q_0 -1}(H)$ such that, $$\eta(B_1)\ge T(H)+\max_{1 \le i_1<\ldots<i_{q_{0}-1} \le s-1} \sum_{j=1}^{q_{0}-1} a_{i_j}=T-1+\max_{1 \le i_1<\ldots<i_{q_{0}-1} \le s-1} \sum_{j=1}^{q_{0}-1} a_{i_j},$$ and with a universal vertex $w$ satisfying ${B_1}_{w,w}\neq 0$. Now, we consider a new matrix $B$ as follows:
 $$B= \left(\begin{array}{@{}cccc|c|c@{}}
     &  &   &   &\overbrace{0 ~~\ldots ~~ 0}^{k_s} & b_1 \\
     &  & B_1~~~~  &   &\vdots &\vdots  \\
     &  &   &   & 0~~\ldots ~~ 0 &\vdots  \\
     &  &   &   & 0 ~~\ldots ~~0 & b_k \\\hline
    0 & \ldots & 0  & 0& 0 ~~~~~~~~~~~~ & b_{k+1} \\
    \vdots & \ldots & \vdots  & \vdots & \ddots & \vdots \\
   0 & \ldots & 0  & 0& ~~~~~~~~~~~~0 & b_{n-1} \\\hline
    b_1 & \ldots & \ldots  & b_k & b_{k+1}~\ldots~b_{n-1}  & c
  \end{array}\right),$$ where $c, b_i\neq 0$ for $1\le i \le n-1$. We have  
  $\rank(B)=\rank(B_1)+2$, that is, $\eta(B)=\eta(B_1)+k_s-1$. This combined with Cauchy's interlacing theorem (see \cite{HJ1}) applied to $B$ and $B(w)$, gives $B\in S_{q_0}(G)$. Moreover, from the above we have 
 \begin{eqnarray*}
 \eta(B) & \ge & T-1+\max_{1 \le i_1<\ldots<i_{q_{0}-1} \le s-1} \sum_{j=1}^{q_{0}-1} a_{i_j}+k_s-1 \\
 &= & T+\max_{1 \le i_1<\ldots<i_{q_{0}-1} \le s-1} \sum_{j=1}^{q_{0}-1} a_{i_j}+k_s-2 \\
 &=& 
 T+\max_{1 \le i_1<\ldots<i_{q_{0}-1} \le s-1} \sum_{j=1}^{q_{0}-1} a_{i_j}+a_s, \end{eqnarray*}
  and $B_{w,w}=c\neq 0$.  \end{pf}

We now complete the proof of Lemma \ref{thm:lowerbound on z_q}.

\noindent {\bf Proof of Lemma \ref{thm:lowerbound on z_q}:} 
 The proof of this result will use induction on $q$ and has been verified for $q=0$. So, consider the case in which $q=1$. To establish the desired inequality for $q=1$, we use induction on the trace $T$ (observe that $q \leq s \leq T$). 
\vspace{3mm}

(Base of Induction) If $T=1~(=q)$, then  $G\cong (0^{(n-1)},1^{(1)})$. Considering the $n \times (n-1)$ matrix $M$ given by
$$M=\left(
  \begin{array}{cccc}
    1 &   1  & \ldots & 1  \\
    1 &   0    & \ldots & 0  \\
    0 &   1    & \ddots &  \vdots  \\
    \vdots &  \ddots    & \ddots &  0  \\
    0 &   0    & \ldots &  1 \\
  \end{array}
\right).$$ Then, $M^T\,M=I_{n-1}+J_{n-1}$ and its eigenvalues are denoted by $\sigma(M^T\,M)=\{n,\, 1^{(n-2)}\}$ and consequently, $\sigma(M\,M^T)=\{n,\, 1^{(n-2)},0\}$. Set $A=M\,M^T-I_n$. Then we have $A\in S_1(G)$ with $\eta(A)=n-2=T+a_1$ and $A_{w,w}=n-2$. 

\vspace{3mm}

(Induction Hypothesis) Assume that for any connected threshold graph $G$ with trace less than $T$ equation (\ref{lower on z_q}) is true. 
 
\vspace{3mm} 
 
(Induction Step) We prove that equation (\ref{lower on z_q}) is true for any connected  graph $G\cong (0^{(k_1)}, 1^{(t_1)}, \ldots,0^{(k_s)}, 1^{(t_s)})$ with trace  $T$. We consider the following two cases:

\vspace{3mm} 

$Case\, 1)$ $t_s\ge 2$. Let $H\cong (0^{(k_1)}, 1^{(t_1)}, \ldots,0^{(k_s)}, 1^{(t_{s}-1)})$. Using the induction hypothesis we have 
 $\exists B\in S_1(H)$ such that
 \begin{equation}\label{z_q1}
\eta(B)\ge T-1+\max_{1\le i \le s}a_i.
 \end{equation}
  and at least one universal vertex $w$ satisfies $B_{w,w}\neq 0$. Then, matrix $B$ has the following form:
  
   $$B= \left(\begin{array}{@{}ccc|c@{}}
    * & \ldots & *& b_1  \\
    \vdots & \ddots & \vdots & \vdots  \\
    * & \ldots & * & b_{n-2}  \\\hline
    b_1 & \ldots & b_{n-2} & b_{n-1}  \\
  \end{array}\right),$$  where the last row and column correspond to the vertex $w$, which gives,  $b_i\neq 0$ for $i\in \{1,2,\ldots, n-1\}$. Now, using the matrix $B$, we define the matrix $B^{\prime}$ as follows: 
 $$B^{\prime}= \left(\begin{array}{@{}ccc|c|c@{}}
    * & \ldots & *& b_1 & b_1 \\
    \vdots & \ddots & \vdots & \vdots & \vdots \\
    * & \ldots & * & b_{n-2} & b_{n-2} \\\hline
    b_1 & \ldots & b_{n-2} & b_{n-1} & b_{n-1} \\\hline
    b_1 & \ldots & b_{n-2} & b_{n-1} & b_{n-1} 
  \end{array}\right).$$

Obviously, $\rank(B^{\prime})=\rank(B)$. This combined with Cauchy's interlacing theorem for $B$ and $B^{\prime}$ gives $B^{\prime}\in S_1(G)$ and
 \begin{equation}\nonumber
\eta(B^{\prime})=\eta(B)+1\ge T+\max_{1\le i \le s}a_i.
 \end{equation}
 Moreover, corresponding to the last universal vertex $w^{\prime}$ in $G$, we have $B^{\prime}_{w^{\prime}, w^{\prime}}=b_{n-1}\neq 0$. This completes the induction in this case. 

\vspace{3mm} 

$Case\, 2)$ $t_s=1$. Let $H\cong (0^{(k_1)}, 1^{(t_1)}, \ldots,0^{(k_{s-1})}, 1^{(t_{s-1})})$. By the induction hypothesis we have 
 $\exists B\in S_1(H)$ such that
 \begin{equation}\label{z_q1-2}
\eta(B)\ge T-1+\max_{1\le i \le s-1}a_i.
 \end{equation}
  and at least one universal vertex $w$ satisfies $B_{w,w}\neq 0$. Then matrix $B$ has the following form:
  
   $$B= \left(\begin{array}{@{}ccc|c@{}}
    * & \ldots & *& b_1  \\
    \vdots & \ddots & \vdots & \vdots  \\
    * & \ldots & * & b_{k-1}  \\\hline
    b_1 & \ldots & b_{k-1} & b_{k}  \\
  \end{array}\right),$$  where $k=n-1-k_s$ and the last row and column correspond to the vertex $w$, which gives,  $b_i\neq 0$ for $i\in \{1,2,\ldots, k\}$. Now, using the matrix $B$, define the matrix $B^{\prime}$ as follows: 
 $$B^{\prime}= \left(\begin{array}{@{}ccc|c|c|c@{}}
    * & \ldots & *& b_1 &\overbrace{0 ~\ldots ~ 0}^{k_s} & b_1 \\
    \vdots & \ddots & \vdots & \vdots &\vdots & \vdots \\
    * & \ldots & * & b_{k-1}  & 0~ \ldots ~ 0 &b_{k-1} \\\hline
    b_1 & \ldots & b_{k-1}  & b_k& 0 ~\ldots ~0 & b_k \\\hline
    0 & \ldots & 0  & 0& a ~~~~~~~~~ & a \\
    \vdots & \ldots & \vdots  & \vdots & \ddots & \vdots \\
   0 & \ldots & 0  & 0& ~~~~~~~~~~a & a \\\hline
    b_1 & \ldots & b_{k-1} & b_k & a~\ldots~a  & c
  \end{array}\right),$$ where $c=b_k+k_s\,a$. We choose a positive value for $a$ such that   $a\neq \frac{-b_k}{k_s}$. 

Obviously, $\rank(B^{\prime})=\rank(B)+k_s$. This combined with Cauchy's interlacing theorem for $B^{\prime}$ and $B^{\prime}(n)$ gives $B^{\prime}\in S_1(G)$ and
 \begin{equation}\nonumber
\eta(B^{\prime})=\eta(B)+1\ge T+\max_{1\le i \le s-1}a_i.
 \end{equation} 
 At this stage, if $a_s$ is not maximum among the $a_i$'s, then we are done. If this is not that case, we may use Lemma \ref{ts=1} (note that the desired lower bounds holds for $q=0$) to complete the proof in this case, so
 \begin{equation}\label{z_q3}
 \eta(B^{\prime})\ge T+\max_{1\le i \le s}a_i.
 \end{equation}
 
 Moreover, corresponding to the last (universal) vertex $w^{\prime}$ in $G$, we have $B^{\prime}_{w^{\prime}, w^{\prime}}=c\neq 0$. This completes the induction on $T$ when $q=1$ and hence verifies the base case for the first induction on $q$. Now, assume that Lemma \ref{thm:lowerbound on z_q} holds for all $q <q_0$.
 Let $G\cong (0^{(k_1)}, 1^{(t_1)}, \ldots,0^{(k_s)}, 1^{(t_s)})$ with trace $T=\sum_{i=1}^s t_i$ and let $a_j=\max\{k_j-2,\,0\}$ for $j\in \{1,2,\ldots,s\}$, and $T \geq s \geq q_0$. To verify the desired inequality for $q_0$, we will use induction on $T$ with $q_0$ fixed.
 
 Base case: $T=q_0=s$. Let $G\cong (0^{(k_1)}, 1^{(t_1)}, \ldots,0^{(k_s)}, 1^{(1)})$. Consider the matrix $M$ given by
  $$M=\left(\begin{array}{@{}c|c|c|c@{}}
    {\cal J}_1 & {\cal J}_2 &\ldots & {\cal J}_{s} \\\hline
    I_{k_1} & O & \ldots & O \\\hline
    O  &  I_{k_2} & \ddots & \vdots \\\hline
    \vdots & \ddots &\ddots & O\\\hline
        O & \ldots &O & I_{k_s}
  \end{array}\right),$$ where ${\cal J}_i$ is an $s \times k_i$ matrix obtained from all ones matrix of the same order but with all $i-1$ first rows equal to zero, and $O$ is a rectangular zero matrix  with a proper order. 


Then $$M^T\,M=\left( \begin{array}{cccc}
    q_0 J_{k_1}+I_{k_1} & (q_0-1)J_{k_1,k_2} & \cdots & J_{k_1,k_s} \\[2mm]
    (q_0-1)J_{k_2,k_1}  & (q_0-1)J_{k_2}+I_{k_2} & \cdots &  J_{k_2,k_s}  \\
    \vdots & \ddots & \ddots & \vdots \\
     J_{k_s,k_1}  & J_{k_s,k_2} & \cdots & J_{k_s}+I_{k_s}
  \end{array}
\right),$$ 
where $J_{k,l}$ is the $k \times l$ matrix of all ones.

It is not difficult to show that there exists an invertible matrix $S$ such that  $S(M^T\,M-I)S^T$ is equal to the positive semidefinite matrix given by
$$P=\left( \begin{array}{cccc}
    q_0 J_{k_1} & 0 & \cdots & 0 \\[2mm]
    0  & \frac{(q_0-1)}{q_0}J_{k_2} & \cdots &  0  \\
    \vdots & \ddots & \ddots & \vdots \\
     0  & 0 & \cdots & \frac{q_0-s+1}{q_0-s+2} J_{k_s}
  \end{array}
\right).$$ 
Observe that the rank of $P$ is $s$. Since $S$ is invertible (or we can apply the classical Sylvester's Law of inertia), it follows that the
spectrum of $M^T\,M$ is $\sigma(M^T\,M)=\{1^{(\sum k_i-s)}, \lambda_1, \lambda_2, \ldots, \lambda_s\}$ , where $\lambda_1, \lambda_2, \ldots, \lambda_s$ are all more than 1.
Define the matrix $A=-(M\,M^T-I)$. We have $A\in S_{q_0}(G)$ with $\eta(A)=k_1+k_2+\cdots + k_s-s=T+a_1+a_2+\cdots + a_s$ and  $A_{w,w}\neq 0$. 
\vspace{3mm}

Now, assume that Lemma \ref{thm:lowerbound on z_q} holds for fixed $q_0$ and for all such connected threshold graphs $G$ with trace $T < T_0$. 

Let $G\cong (0^{(k_1)}, 1^{(t_1)}, \ldots,0^{(k_s)}, 1^{(t_s)})$ with trace $T_0 \geq q_0$. The remainder of the proof relies on two cases.

\vspace{3mm} 

$Case\, 1)$ $t_s\ge 2$. Let $H\cong (0^{(k_1)}, 1^{(t_1)}, \ldots,0^{(k_s)}, 1^{(t_{s}-1)})$. By the induction hypothesis on $T_0 -1$ we know 
 $\exists B\in S_{q_0}(H)$ such that,
 \begin{equation}\label{}
 \eta(B)\ge T_0-1+\max_{1 \le i_1<\ldots<i_{q_0} \le s} \sum_{j=1}^{q_0} a_{i_j},
 \end{equation}
  and at least one universal vertex $w$ satisfies $B_{ww}\neq 0$. Then, matrix $B$ has the following form:
  
   $$B= \left(\begin{array}{@{}ccc|c@{}}
    * & \ldots & *& b_1  \\
    \vdots & \ddots & \vdots & \vdots  \\
    * & \ldots & * & b_{n-2}  \\\hline
    b_1 & \ldots & b_{n-2} & b_{n-1}  \\
  \end{array}\right),$$  where the last row and column correspond to the vertex $w$, which gives,  $b_i\neq 0$ for $i\in \{1,2,\ldots, n-1\}$. Now, using the matrix $B$, define $B^{\prime}$ as follows: 
 $$B^{\prime}= \left(\begin{array}{@{}ccc|c|c@{}}
    * & \ldots & *& b_1 & b_1 \\
    \vdots & \ddots & \vdots & \vdots & \vdots \\
    * & \ldots & * & b_{n-2} & b_{n-2} \\\hline
    b_1 & \ldots & b_{n-2} & b_{n-1} & b_{n-1} \\\hline
    b_1 & \ldots & b_{n-2} & b_{n-1} & b_{n-1} 
  \end{array}\right).$$

Obviously, $\rank(B^{\prime})=\rank(B)$. This combined with Cauchy's interlacing theorem for $B$ and $B^{\prime}$ gives $B^{\prime}\in S_{q_0}(G)$ and
 \begin{equation}\nonumber
\eta(B^{\prime})=\eta(B)+1\ge T_0+\max_{1 \le i_1<\ldots<i_{q_0} \le s} \sum_{j=1}^{q_0} a_{i_j}.
 \end{equation}
 Moreover, corresponding to the last (universal) vertex $w^{\prime}$ in $G$, we have $B^{\prime}_{w^{\prime}, w^{\prime}}=b_{n-1}\neq 0$. This completes the induction in this case. 

\vspace{3mm} 

$Case\, 2)$ $t_s=1$. Let $H\cong (0^{(k_1)}, 1^{(t_1)}, \ldots,0^{(k_{s-1})}, 1^{(t_{s-1})})$. By the induction hypothesis on $T_0 -1$, we know 
 $\exists B\in S_{q_0}(H)$ such that,
 \begin{equation}\label{}
\eta(B)\ge T_0-1+\max_{1 \le i_1<\ldots<i_{q_0} \le s-1} \sum_{j=1}^{q_0} a_{i_j},
 \end{equation}
  and at least one universal vertex $w$ satisfies $B_{w,w}\neq 0$. Then, matrix $B$ has the following form:
  
   $$B= \left(\begin{array}{@{}ccc|c@{}}
    * & \ldots & *& b_1  \\
    \vdots & \ddots & \vdots & \vdots  \\
    * & \ldots & * & b_{k-1}  \\\hline
    b_1 & \ldots & b_{k-1} & b_{k}  \\
  \end{array}\right),$$  where $k=n-1-k_s$ and the last row and column correspond to the vertex $w$, which gives,  $b_i\neq 0$ for $i\in \{1,2,\ldots, k\}$. Now, using the matrix $B$,  define the matrix $B^{\prime}$ as follows: 
 $$B^{\prime}= \left(\begin{array}{@{}ccc|c|c|c@{}}
    * & \ldots & *& b_1 &\overbrace{0 ~\ldots ~ 0}^{k_s} & b_1 \\
    \vdots & \ddots & \vdots & \vdots &\vdots & \vdots \\
    * & \ldots & * & b_{k-1}  & 0~ \ldots ~ 0 &b_{k-1} \\\hline
    b_1 & \ldots & b_{k-1}  & b_k& 0 ~\ldots ~0 & b_k \\\hline
    0 & \ldots & 0  & 0& a ~~~~~~~~~ & a \\
    \vdots & \ldots & \vdots  & \vdots & \ddots & \vdots \\
   0 & \ldots & 0  & 0& ~~~~~~~~~~a & a \\\hline
    b_1 & \ldots & b_{k-1} & b_k & a~\ldots~a  & c
  \end{array}\right),$$ where $c=b_k+k_s\,a$. We choose a positive value for $a$ such that   $a\neq \frac{-b_k}{k_s}$. 

Obviously, $\rank(B^{\prime})=\rank(B)+k_s$. This combined with Cauchy's interlacing theorem for $B^{\prime}$ and $B^{\prime}(n)$ gives $B^{\prime}\in S_{q_0}(G)$ and
 \begin{equation}\nonumber
\eta(B^{\prime})=\eta(B)+1\ge T_0+\max_{1 \le i_1<\ldots<i_{q_0} \le s-1} \sum_{j=1}^{q_0} a_{i_j} 
 \end{equation} 
 At this stage, if $a_s$ is not at least the $q_0^{th}$ maximum among the $a_i$'s, then we are done. If this is not that case, we may use Lemma \ref{ts=1} (note that our conjecture holds for all $q<q_0$ by the inductive hypothesis) to complete the proof in this case.
 
 Hence, we have verified the existence of a $B' \in S_{q_0}(G)$ such that 
 \begin{equation}\label{}
 \eta(B^{\prime})\ge T+\max_{1 \le i_1<\ldots<i_{q_0} \le s} \sum_{j=1}^{q_0} a_{i_j}.
 \end{equation}

This completes the inductive step with respect to $T=T_0$. Therefore, Lemma \ref{thm:lowerbound on z_q} holds for all $T \geq q_0$. Consequently, we have completed the inductive step on $q$ for $q=q_0$. Hence we have completed the proof of Lemma \ref{thm:lowerbound on z_q} for all $1 \leq q \leq s$.  
\newline \hspace*{13cm} $\qed$

\begin{theorem}\label{m_q=z_q}
 Let $G\cong (0^{(k_1)}, 1^{(t_1)}, \ldots,0^{(k_s)}, 1^{(t_s)})$ be a connected threshold graph on $n$ vertices with trace $T=\sum_{i=1}^s t_i$ and let $a_j=\max\{k_j-2,\,0\}$ for $j\in \{1,2,\ldots,s\}$. Then for any $q\in \{0,1,\ldots,s\}$, 
 \begin{equation}\label{m-z_q}
 M_q(G)=Z_q(G)= T+\max_{1 \le i_1<\ldots<i_q \le s} \sum_{j=1}^q a_{i_j}.
 \end{equation}
 \end{theorem}

\begin{pf}
Using Theorem \ref{thm:upperbound on z_q} and Lemma \ref{thm:lowerbound on z_q}, we have 
\[ T+\max_{1 \le i_1<\ldots<i_q \le s} \sum_{j=1}^q a_{i_j} \leq M_q(G) \leq Z_q(G) \leq T+\max_{1 \le i_1<\ldots<i_q \le s} \sum_{j=1}^q a_{i_j}.\]
\end{pf}

 Obviously, the maximum value for $Z_q(G)$ in (\ref{z_q}) occurs when $q=s$. Indeed, when $q=s$, we have $$Z(G)=Z_s(G)=T+\sum_{j=1}^s a_{i}.$$ Assume that $p=|\{j\,:\, j\in \{1,2,\ldots, s\},\,k_j\ge 2\}|$ and let $k_{i_j}\ge 2$ for $j\in \{1,2,\ldots,p\}$. Then we have 
\begin{align}
Z(G)= T+\sum_{j=1}^s a_{i}=T+\sum_{j=1}^p (k_{i_j}-2)=& T+\sum_{j=1}^p k_{i_j}-2p \nonumber\\
 =& T+(\#\, of\, zeros) -(s-p)-2p \nonumber\\
 =& T+n-T-p-s=n-p-s,  \nonumber
\end{align} where $n$ is the number of vertices in $G$.

 From the well-known formula given in (\ref{well-known on Z}) we have for any connected threshold graph $G$
 \begin{align}
 Z(G)=M(G)= n-2T+s_1+2s_0=&~ n-2T+(s-p)+2\,\sum_{j=1}^s (t_j-1)\nonumber\\
                =&~ n-2T+s-p+2T-2s=n-s-p.
 \end{align} 
We summarize the above remarks with the following result.
\begin{cor}
 Let $G\cong (0^{(k_1)}, 1^{(t_1)}, \ldots,0^{(k_s)}, 1^{(t_s)})$ be a connected threshold graph with $n$ vertices and let $p=|\{j\,:\, j\in \{1,2,\ldots,s\},\,k_j\ge 2\}|$. Then 
 $$Z(G)=M(G)=n-s-p.$$ In particular, the minimum rank of a connected threshold graph is equal to $s+p$.
\end{cor}

\section{Computing $Z_{q}(G)$ for Certain Structured Families of Graphs}

In this section, we incorporate existing work on both the zero forcing number and the positive semidefinite zero forcing number in an effort to extend these results to the $q$-analogue of zero forcing for certain families of graphs for $q$ at least one. Many of the results in this section will rely on computing $Z_0$, and we note here the well-known fact that for any graph $G$, $Z_0(G) \geq\kappa(G)$, where 
$\kappa(G)$ is the vertex connectivity of the graph $G$, that is the fewest numbers to be deleted to disconnect $G$ (see, for example, \cite{barioli2013parameters}).

%
	\subsection{Cartesian product}
	For any graph $G$, consider the Cartesian product $G\sqr K_2$. If $V(G) = \{1,2,\ldots,n\}$ and $V(K_2) = \{a,b\}$, then label the vertices of the layers of $G\sqr K_2$ by $1a,2a,\ldots,na$ and $1b,2b,\ldots,nb$.
	We begin this section by considering the following graphs: $K_n \sqr K_2$, the {\em ladder graph} $P_n \sqr K_2$, and {\em prism} $C_n \sqr K_2$.
	\begin{prop} Suppose $n\in \mathbb{N}$.
	\begin{enumerate}
	\item Then $Z_0(K_n \sqr K_2) =  Z_1(K_n \sqr K_2) =\ldots =Z(K_n \sqr K_2) = n$.
	\item For any $n\geq 3$, $Z_0(P_n \sqr K_2) = Z_1(P_n \sqr K_2) = \ldots = Z(P_n \sqr K_2) = 2$.
	\item For any $n\geq 3$, we have $Z_0(C_n \sqr K_2) = Z_1(C_n \sqr K_2) = \ldots = Z(C_n \sqr K_2) = 4$.
	\end{enumerate}
	\end{prop}
	\begin{proof}
	For each of the graphs listed above we make use of previous work to squeeze all of the numbers $Z_q$. For the graphs in item (1), observe that it 
is know from \cite{AIM} that $Z(K_n \sqr K_2)=n$, and from \cite{peters} we also have $Z_0 ( K_n \sqr K_2)=n$, which verifies the equalities in (1). Similar arguments apply to the remaining items using the results from \cite{AIM, peters}.	 
	\end{proof}

	For any $n\geq 4$, we define the {\em book graph} $B_n$  as the graph $K_{1,n} \sqr K_2$.
	\begin{theorem}\label{book}
		For any $n\geq 3$, we have
		\begin{align*}
			Z_q(B_n) = \left\{
			\begin{aligned}
				2 & \mbox{ if } q = 0\\
				n & \mbox{ for }q \geq 1.
			\end{aligned}
			\right. 
		\end{align*}
	\end{theorem}
	\begin{proof}
		Label the leaves of the layer $aK_{1,n}$ by $1a,2a,\ldots,na$ and the central vertex by $ca$. Similarly, label the leaves of the layer $bK_{1,n}$ by $1b,2b,\ldots,nb$ and the central vertex by $cb$. For the $Z_0$-game, place two tokens on $\{ca,cb\}$. We note that $B_n[V(B_n)\setminus \{ca,cb\}]$ is a disjoint union of $n$ edges. We give one of these edges in each round. No matter which edge is returned by the oracle, we are able to use the CCR to force both end points of the edge using the central vertices $ca$ and $cb$. The graph induced by the uncoloured vertices is again a disjoint union of edges and we repeat the previous strategy.  Hence, $Z_0(B_n) \leq 2$. It is obvious that we  need to spend at least two tokens to win the $Z_0$-game since $\kappa(B_n)>1$. Therefore, $Z_0(B_n) = 2$.
		
		Now we consider the $Z$-game. Place $n$ tokens on all vertices of the layer $aK_{1,n}$ except one leaf. Using the CCR, we are able to colour all vertices of $B_n$. Therefore, $Z_1(B_n) \leq Z(B_n)\leq n$.
		
		On the other hand, we can also prove that $Z_1(B_n) \geq n$, which is not trivial. To show this, we find a weighted adjacency matrix $C\in S_1(B_n)$ such that $\operatorname{rank}(C) = n+2$. Hence, we will have $Z_1(B_n) \geq M_1(B_n) \geq n$. 
		
		The adjacency matrix $B$ of $B_n$ is the matrix
		\begin{align*}
			B =\begin{bmatrix}
				A & I_n\\
				I_n & A
			\end{bmatrix},
		\end{align*}
		where $A$ is the adjacency matrix of $K_{1,n}$. The spectrums of $A$ and $B$ are 
		\begin{align*}
			\sigma(A) &= \left\{ \sqrt{n}^{(1)},-\sqrt{n}^{(1)},0^{(n-1)} \right\},\\
			\sigma(B) &= \left\{ \sqrt{n}+1^{(1)}, - \sqrt{n}+1^{(1)}, 1^{(n-1)}, \sqrt{n}-1^{(1)}, -\sqrt{n}-1^{(1)}, -1^{(n-1)} \right\}. 
		\end{align*}
		For any $2\leq i \leq n$, let $u_i = e_i-e_{i+1}$. The vectors $(u_i)_{i = 2,3,\ldots,n}$ form a basis of the eigenspace (the nullspace) of $A$. The eigenspace corresponding to the eigenvalue $\sqrt{n}$ is spanned by the vector $v_1 = \left( 1, \frac{\sqrt{n-1}}{n-1}, \frac{\sqrt{n-1}}{n-1},\ldots, \frac{\sqrt{n-1}}{n-1} \right)^T$. Since $K_{1,n}$ is bipartite, the generator of the eigenspace of $-\sqrt{n}$ can be easily obtained from $v_1$. We denote this vector by $v_2$. Let $U = \begin{bmatrix}
		1 \\
		1
		\end{bmatrix}$ and $V = \begin{bmatrix}
		1\\
		-1
		\end{bmatrix}$. The eigenvectors of $B$ are of the form
		\begin{align*}
			U
			\otimes 
			x \mbox{ and }
			V
			\otimes 
			x,
		\end{align*}
		for any $x\in \{v_1,v_2,u_2,\ldots,u_n\}$, and where $\otimes$ denotes the tensor product of matrices.
		
		Consider the matrix 
		\begin{align*}
			D = 
			\begin{bmatrix}
				-\frac{\sqrt{n}}{2} I_n & \left( \frac{-\sqrt{n}}{2} +1\right) I_n\\
				\left( -\frac{\sqrt{n}}{2} +1\right) I_n & -\frac{\sqrt{n}}{2} I_n
			\end{bmatrix} = 
			\begin{bmatrix}
			-\frac{\sqrt{n}}{2}  & \left( -\frac{\sqrt{n}}{2} +1\right) \\
			\left( -\frac{\sqrt{n}}{2} +1\right) & -\frac{\sqrt{n}}{2}
			\end{bmatrix} \otimes I_n,
		\end{align*}
		 and let $C = B - D$.	We prove $C\in S_1(B_n)$ and has nullity equal to $n$. 
		 
		 From the structure of $B$, the addition of the matrix $D$ does not annihilate the non-zero off diagonal entries of $-B$. Moreover, the zero entries of $B$ do not change. Hence, $C \in S(B_n)$. 
		 
		 Observe that the matrix $D$ commutes with $B$. Hence $D$ and $B$ are simultaneously diagonalizable. Moreover, 
		 \begin{align*}
		 	D (U\otimes v_1) &= (-\sqrt{n} +1)U\otimes v_1,\\
		 	D (U\otimes v_2) &= (-\sqrt{n} +1)U\otimes v_2,\\
		 	D (U \otimes u_i) &= (-\sqrt{n} +1)(U \otimes u_i), \mbox{ for any } 2\leq i\leq n\\
		 	D (V\otimes v_1) &= -V\otimes v_1,\\
		 	D (V\otimes v_2) &= -V\otimes v_2,\\
		 	D (V \otimes u_i) &= -(V \otimes u_i), \mbox{ for any } 2\leq i\leq n.
		 \end{align*}
		 The spectrum of $D$ is $\sigma(D) = \left\{ -1^{(n+1)}, -\sqrt{n}+1^{(n+1)} \right\}$. We conclude that $-1 - (-1) = 0$ is an eigenvalue of $C$ with eigenvector $V\otimes u_i$, for $2\leq i \leq n$. Similarly, $-\sqrt{n}+1 - \left(-\sqrt{n} + 1\right) = 0$ is an eigenvalue of $C$ with eigenvector $U\otimes v_2$. Upon a careful analysis of the corresponding eigenvectors, we conclude that the spectrum of $C$ is $$\sigma(C) = \left\{ 2\sqrt{n}^{(1)},0^{(n)}, \sqrt{n}^{(n)}, -\sqrt{n}^{(1)} \right\}.$$
		 
		Therefore, the nullity of $C$ is $n$ and $C$ has exactly one negative eigenvalue, which is $-\sqrt{n}$. Hence, $C\in S_1(B_n)$. Therefore, $M_1(B_n) \geq n$. Consequently, $Z_1(B_n)\geq n$. This completes the proof. 
	\end{proof}
	
	\subsection{Complete bipartite graphs}
The complete bipartite graph is denoted as $K_{n,m}$, for some $n,m\geq 2$, and assume the vertices of $K_{n,m}$ are partitioned as $V_n \cup V_m$, where $|V_n |=n$ and $|V_m |=m$. Since $(n,m) \neq (1,1)$, the adjacency matrix associated with $K_{n,m}$ has $3$ distinct eigenvalues. In fact, if $A$ is the adjacency matrix for $K_{n,m}$, we have 
	\begin{align*}
		\sigma(A) = \left\{ \sqrt{nm}^{(1)}, -\sqrt{nm}^{(1)}, 0^{(n+m-2)} \right\}.
	\end{align*}
	
	\begin{prop}
		For any $n,m\geq 1$
		\begin{align*}
			Z_q(K_{n,m}) &= 
			\left\{ 
			\begin{aligned}
			&\min(n,m)  &\mbox{ if } q = 0\\
			&n+m-2 &\mbox{ for } q\geq 1.
			\end{aligned}
			\right. 
		\end{align*}
		
	\end{prop}
	\begin{proof}
	From \cite[Prop. 3.4]{peters} we have  $Z_0(K_{n,m}) = n$ if $n \leq m$. Furthermore, we also know from \cite{AIM} that $Z(K_{n,m})=n+m-2$.
Now, we consider the $Z_1$-game. The nullity of the graph $K_{n,m}$ is equal to the multiplicity of the eigenvalue $0$ in the adjacency matrix for $K_{n,m}$, call it $A$. Therefore, $n+m-2 \leq M_1(A) \leq Z_1(K_{n,m})$. It suffices to prove that $Z(K_{n,m}) \leq n+m-2$ to complete the proof. It is not hard to see that any set $S\cup T$, where $S\subset V_n$ and $T \subset V_m$ of size $n-1$ and $m-1$, respectively, is a zero forcing set of $K_{n,m}$. Hence, $Z(K_{n,m}) \leq n-1 +m-1 = n+m- 2$. This completes the proof.
	\end{proof}
	
	Consider the graph $Y = K_{n,m} \sqr K_2,$ and let $A$ be the adjacency matrix for $Y$.
	Let $(u_i)_{i=1,\ldots,n+m-2}$ be a basis of the eigenspace for $A$ corresponding to $0$. Let $v_1$ and $v_2$ be the generators for the eigenspaces  for $A$ of $\sqrt{nm}$ and $-\sqrt{nm}$, respectively. 	The spectrum of $A$ is
	\begin{align*}
	\sigma(A) &= \left\{ \sqrt{nm}+1^{(1)}, -\sqrt{nm}+1^{(1)}, 1^{(n+m-2)},\sqrt{nm}-1^{(1)}, -\sqrt{nm}-1^{(1)}, -1^{(n+m-2)} \right\}.
	\end{align*}
	We formulate the following conjecture regarding the graph $Y=K_{n,m} \sqr K_2$.
	\begin{conj}
		For any $n,m\geq 2$,
		\begin{align*}
			Z_q(K_{n,m} \sqr K_2) &= 
			\left\{ 
				\begin{aligned}
					&2\min(n,m)  &\mbox{ if } q = 0\\
					&n+m &\mbox{ for } q\geq 1.
				\end{aligned}
			\right. 
		\end{align*}
	\end{conj}
Currently, we have not completely resolved the above conjecture, but we are able to provide some fairly compelling bounds. Consider the following argument.
	Without loss of generality, assume that $n<m$. First, we prove that $Z_0(Y) = 2n$. Place $2n$ tokens on the two copies of the bipartition of $K_{n,m}$ in $Y$. Let $S$ be the set of these vertices. The graph $Y[V(Y)\setminus S]$ is a disjoint union of edges. It is straightforward that we can complete the $Z_0$-game by returning an endpoint of each such edge, in each round. Hence, $Z_0(Y) \leq 2n$.

	Consider the matrix, 
	$$C = \begin{bmatrix}
	\frac{\sqrt{nm}}{2} & \frac{\sqrt{nm}}{2}-1\\
	\frac{\sqrt{nm}}{2}-1 & \frac{\sqrt{nm}}{2}
	\end{bmatrix} \otimes I_{n+m} = \begin{bmatrix}
	\frac{\sqrt{nm}}{2} & 0\\
	0 & \frac{\sqrt{nm}}{2}
	\end{bmatrix} \otimes I_{n+m} + \begin{bmatrix}
	0 & \frac{\sqrt{nm}}{2}-1\\
	\frac{\sqrt{nm}}{2}-1 & 0
	\end{bmatrix} \otimes I_{n+m}.$$ 
	
	The spectrum of the matrix $C$ is
	\begin{align*}
	\sigma(C) &= \left\{ \sqrt{nm}-1^{(n+m)},1^{(n+m)} \right\}.
	\end{align*}
	
	Applying a similar argument as in the proof of Theorem \ref{book} we can show that there is a matrix $A$ in $S_1(Y)$ with nullity $n+m-1$. Hence  $Z_1(Y) \geq n+m-1$. 
	Let $S$ be the vertices corresponding to a layer of $K_{n,m}$ in $Y$. It is easy to see that $S$ is a zero-forcing set in the Cartesian product $K_{n,m} \sqr K_2$. Hence, $Z(Y) \leq n+m$. We conclude that $Z_q(Y) \in \{n+m-1,n+m\}$, for any $q\geq 1$.
	
%
%
	
	\subsection{Strongly regular graphs}
	Let $G$ be a strongly regular graph with parameter $(n,k,\lambda,\mu)$ and assume that $\sigma(A) = \{ k^{(1)}, \theta^{(n-1-\ell)}, \tau^{(\ell)} \}$, where $\tau <0$ and $A$ is the adjacency matrix of $G$, see \cite{GM} for more details.
	\begin{prop} If $G$ is a strongly regular graph with $\ell$ negative eigenvalues, then
		$Z_0(G) \geq \ell$ and $Z_1(G)\geq n-1-\ell$.\label{lem:lower-bound-srg}
	\end{prop}

	\begin{proof}
		Let $B = A-\tau I  $. It is obvious that $B$ belongs to $S_0(G)$. Moreover, every eigenvalue of $B$ is nonnegative and $0$ is an eigenvalue with multiplicity $\ell$. Hence, $Z_0(G) \geq M_0(G) \geq \ell$. Similarly, let $C=-A+\theta I$. Then $C \in S_{1}(G)$ and the nullity of $C$ is $n-1-\ell.$ Thus $Z_1(G) \geq M_1(G) \geq n-1-\ell$.
	\end{proof}
	
	The bounds in Proposition~\ref{lem:lower-bound-srg} are sharp, as shown in the following example.
	\begin{exam}
		Consider the Petersen graph denoted by $G$. Then $G$ is strongly regular with spectrum $\sigma(A) = \{3^{(1)},1^{(5)},-2^{(4)}\}$. By Proposition~\ref{lem:lower-bound-srg}, we have $Z_0(G)\geq 4$ and $Z_1(G) \geq 5$.  It is not difficult to verify that $Z_0(G) = 4$ and since $Z(G)=5$, it follows that $Z_1(G) = 5$.
	\end{exam}
	\begin{question}
		What can we say about $Z_0$ and $Z_1$ for strongly regular graphs in general?.
	\end{question}
	
	We consider a special case of the above question below.
	
	\subsubsection{Complete multipartite graphs}
	Let $G_{n,\ell} := K_{\underbrace{n,n,\ldots,n}_{\ell}}$. The spectrum of $A$ is
	\begin{align*}
	\sigma(A) = \left\{ n(\ell-1)^{(1)},-n^{(\ell-1)}, 0^{(\ell(n-1))} \right\},
	\end{align*} where $A$ is the adjacency matrix for $G_{n,\ell}$.

The following is straightforward.
\begin{obs}
		Let $n\geq 2$, $\ell \geq 3$ and $S\subset V(G_{n,\ell})$. Then, $G_{n,\ell}[S]$ is connected if and only if $S$ is not contained in a single part of $G_{n,\ell}$. Moreover, if $G_{n,\ell}[S]$ is disconnected, then it is a union of isolated vertices. \label{lem:cmp-connectedness}
	\end{obs}
	\begin{prop}\cite{AIM}
		For any $n\geq 1$ and $\ell \geq 3$, $Z(G_{n,\ell}) = n\ell -2$.\label{prop:cmp-zero-forcing}
	\end{prop}
	
	For general $q$, we formulate the following conjecture concerning the complete multipartite graph.
	
	\begin{conj}
		For any $n\geq 1$ and $\ell  \geq 3$, 
		\begin{align*}
			Z_q(G_{n,\ell})
			&=
			\left\{
				\begin{aligned}
					&n(\ell-1) &\mbox{ if }q=0,\\
					&n\ell -2 & \mbox{ if }q\geq 1.
				\end{aligned}
			\right.
		\end{align*}
		
	\end{conj}

	 	First, observe that $Z_0(G_{n,\ell}) = n(\ell-1)$ follows from the work in \cite{peters} and from Proposition \ref{prop:cmp-zero-forcing} $Z(G_{n,\ell})=n\ell -2$.
	 	A natural next step would be to verify that  $Z_1(G_{n,\ell}) \geq n\ell - 2$. One potential argument may go as follows. Suppose that $Z_1(G_{n,\ell}) = t <n\ell - 2$. Assume that the $t$ tokens are placed on $S$ and that there is a winning strategy with these $t$ tokens. Since $t\leq n\ell -3$, there are  $n\ell-t \geq 3$ uncoloured vertices in $G_{n,\ell}$. 
	 	We distinguish the following cases whether these vertices are in a single part or not.
	 	\\ \\
	 	{\it Case~1. All uncoloured vertices $S$ are in a partite set $V_i$, for some $i\in \{1,2,\ldots,\ell\}$}.
	 	
	 	Since all uncoloured vertices are in a single part of $G_{n,\ell}$, the subgraph $G_{n,\ell} \left[ S \right]$ is a union of isolated vertices (see Observation ~\ref{lem:cmp-connectedness}). Since we must provide at least $2$ of these isolated vertices to the oracle, the oracle only needs to return $2$ vertices to ensure we  cannot proceed. So there is no optimal strategy in this case.
	 	\\ \\
	 	{\it Case~2. Two uncoloured vertices are in different parts.} This case currently remains unresolved.
	 	
	 %
	 %

	\subsection{Graphs in the triangular association scheme}
	Let $G_n$ be the Kneser graph $K(n,2)$, for $n\geq 5$. In this section, we prove a result on the $Z_0(G_n)$.
	
	First, we recall the following result by Bre\v{s}ar et al \cite{brevsar2019grundy}.
	
	\begin{theorem}
		\hfil
		\begin{enumerate}
			\item For $r\geq 2$ and $2r+1 \leq n \leq 3r$, $Z(K(n,k)) \leq \binom{n}{2} - \binom{2r}{3} + \binom{4r-1-n}{3r-n}$.
			\item For $r\geq 2$ and $n\geq 3r+1$, $Z(K(n,r))= \binom{n}{r}- \binom{2r}{r}$.
		\end{enumerate}
	\end{theorem}

	\begin{cor}
		For $n\geq 7$, $Z(K(n,2)) = \binom{n}{2} - 6$. Moreover, $Z(K(5,2)) = 5$ and $Z(K(6,2)) = 10$.
	\end{cor}

	It is well-known that the eigenvalues of the adjacency matrix $A$ of $G_n$  are given by $\{\binom{n-2}{2}^{(1)}, (-n-3)^{(n-1)}, 1^{(\binom{n}{2}-n)}\}$ (see, for example, \cite{GM}). Setting $B = -A + \binom{n-2}{2}I$ we have $B \in S_{1}(G_n)$ and the nullity of $B$ is equal to  $\binom{n}{2}-n = \binom{n-1}{2}-1$. Hence $Z_1(G_n) \geq M_{1}(G_n) \geq  \binom{n-1}{2}-1$.
Moreover, if we place $\binom{n-1}{2}$ tokens on all vertices of $G_{n-1}$,
then it is easy to see that we can force all vertices in the $Z_1$-game. Therefore,
	\begin{align*}
		\binom{n-1}{2}-1 \leq Z_1(G_n) \leq \binom{n-1}{2}.
	\end{align*}
	
	We conjecture the following.
	\begin{conj}
		For any $n\geq 8$, $Z_0(G_n) = \binom{n-1}{2}$. Moreover, $Z_0(G_n) = \binom{n}{2}-6$ for $n\in \{5,6,7\}$. Further, $Z_0(G_n) = Z_1(G_n)$ for any $n\geq 7$ and $Z_1(G_n) = Z_0(G_n) +1$ for $n\in \{5,6\}$.
	\end{conj}
	
	To attempt this, we consider the following lemmas. 
	\begin{lemma}
		If $G_n[V(G_n) \setminus S]$ has $k\geq 4$ components for $S \subseteq V(G_n)$, then each component is an isolated vertex.\label{lem:kneser1}
	\end{lemma}
	\begin{proof}
		Suppose that there are $k\geq 4$ components in $G_n[V(G_n) \setminus S]$. Let $C_1,C_2,\ldots,C_k$ be these $k$ components. Without loss of generality, assume that the edge $a$ consisting of the vertices $\{1,2\}$ and $\{3,4\}$ belongs to $C_1$. For any vertex $\{x,y\} \in C_i$ for  $i\geq 2$, we must have that $ \{x,y\}$ is not incident to the edge $a$. Therefore, $\{x,y\}$ is one of $\{1,3\}, \{1,4\}, \{2,3\}$ and $\{2,4\}$. However, the vertices $\{1,3\}, \{1,4\}, \{2,3\}$ and $\{2,4\}$ only produce at most two components. With the component $C_1$, we have at most $k\leq 3$ components, which is a contradiction. Hence, $C_i$ cannot have an edge, for any $i\in \{1,2,\ldots,k\}$.
	\end{proof}

	\begin{lemma}
		If $G_n[V(G_n) \setminus S]$ has at most three components, $C_1,\ C_2, \mbox{ and }C_3$, and if $C_1$ has at least three vertices, then $|C_2| = |C_3| = 1$.\label{lem:kneser2}
	\end{lemma}
	\begin{proof}
		Without loss of generality, suppose that the edge $a = \{ \{1,2\},\{3,4\} \}$.
	\end{proof}

	\begin{lemma}
		If $G_n[V(G_n) \setminus S]$ has at most three components, $C_1,\ C_2, \mbox{ and }C_3$, and if $C_1$ has at least three vertices, then the subgraph induced by $C_1$ must be a star.\label{lem:kneser4}
	\end{lemma}
	\begin{proof}
		By Lemma~\ref{lem:kneser2}, $G_n[V(G_n) \setminus S]$ has three components, two of which are of size $1$. Let $x$ and $y$ be the isolated vertices components. Without loss of generality, assume that the edge joining $\{1,2\}$ and $\{3,4\}$ is in $C_1$. Since $x\cap \{1,2\}\neq \varnothing$ and $x\cap \{3,4\} \neq \varnothing$, any third vertex $z \in C_1$ must be such that $z\cap x \neq \varnothing$ and $z$ is adjacent to $\{1,2\}$ or $\{3,4\}$. If $z$ is adjacent to both $\{1,2\}$ and $\{3,4\}$, then $z\cap \{1,2,3,4\} = \varnothing$, so $z\cap x = \varnothing$. That is, $x$ is adjacent to $z$. Thus, $z$ is only adjacent to one of $\{1,2\}$ or $\{3,4\}$. Assume that $z$ is adjacent to $\{1,2\}$. Then, $\{x,y,z,\{3,4\}\}$ is a coclique (i.e., intersecting) which is contained in a maximum intersecting family (due to the Hilton-Milner theorem, see \cite{GM, hilton1967some}). Without loss of generality, we can assume that $\{x,y,z,\{3,4\}\} \subset $ $\left\{ \{3,x\} \mid x\in \{1,2,\ldots,n\}, \ x\neq 3 \right\}$. Then, it is easy to see that $x = \{1,3\}$ and $y = \{2,3\}$.
		
		If $z^\prime \in C_1$ is not equal to the other three vertices, then $z^\prime \cap x\neq \varnothing$,  $z^\prime \cap y \neq \varnothing$ and $z^\prime \neq \{1,2\}$, so $3\in z^\prime$. This proves that $C_1$ must be a star.
	\end{proof}
	\begin{obs}
		Let $i\in \{1,2,\ldots,n\}$. If $x\in V(G_n)$ such that $i\not \in x$, then $x$ is adjacent to $n-3$ vertices containing $i$.\label{lem:kneser5}
	\end{obs}
	
	\begin{lemma}
		If $G_n[V(G_n) \setminus S]$ is a coclique of size at least $4$, then there exists $i$ such that $S \subset \left\{ \{x,y\} \mid x,y\in \{1,2,\ldots,n\}\setminus \{i\}, x\neq y \right\}$. Equivalently, $G_n[S] $ contains a subgraph which is a copy of $G_{n-1}$.\label{lem:kneser3}
	\end{lemma}
	\begin{proof}
		Let $S$ be a subset of vertices whose removal make $G_n$ be disconnected into isolated vertices of size at least $4$. Let $H = G_n[V(G_n) \setminus S]$. Since there are at least $4$ vertices in $H$, $H$ cannot be a coclique from a copy of $K(4,2)$. Moreover, the Hilton-Milner theorem \cite{GM, hilton1967some} asserts that a maximal cocliques which is not maximum in $G_n = K(n,2)$ has size at most $3$. Thus, the vertices of $H$ are contained in a maximum coclique. Suppose without loss of generality that $H$ consists of the subset of vertices containing $n$. Therefore, $G_n[S]$ contains a copy of $G_{n-1}$. This completes the proof.
	\end{proof}
	
	Given a graph $G = (V,E)$, we let $\tw(G)$ be the \emph{treewidth} of $G$ (see \cite{barioli2013parameters} for more information).
	\begin{lemma}[\cite{barioli2013parameters}]
		Let $G = (V,E)$ be a graph. Then, $\tw(G) \leq Z_0(G)$.\label{lem:tw-z0-lower-bound}
	\end{lemma}

	\begin{lemma}[\cite{harvey2014treewidth}] \label{twh}
		For any $n\geq 6$, $\tw(G_n) = \binom{n}{2}-1$.
	\end{lemma}
	
	\begin{theorem}
		Let $n\geq 5$ and $q\geq 0$. We have
		\begin{align*}
			Z_q(G_n) =
			\begin{cases}
				\binom{n}{2}-6  \hspace*{4cm} &\mbox{ for } n \in \{5,6,7\} \mbox{ and } q=0\\
				\binom{n}{2}-5 &\mbox{ for } n \in \{5,6\} \mbox{ and } q = 1\\
				15 & \mbox{ for } q = 1 \mbox{ and } n = 7\\
				 \binom{n-1}{2}-1 \mbox{ or } \binom{n-1}{2}  & \mbox{ for } q\in \{0,1,2\},\ n\geq 8.\\
			\end{cases}
		\end{align*}\label{lem:tw}
	\end{theorem}
	\begin{proof}
		For $q \in \{0,1\}$ and $n\in \{5,6,7\}$, we can directly compute the values of $Z_q(G_n)$.
		For any $n\geq 8$,  we can apply Lemmas ~\ref{twh} and ~\ref{lem:tw-z0-lower-bound}, to conclude that $Z_0(G_n) \geq \binom{n}{2}-1$. Now, we prove that $Z_2(G_n) \leq \binom{n-1}{2}$. 
		Let $S$ be the collection of all pairs that contain $n$. Now, place $\binom{n-1}{2}$ tokens on the graph $G_{n-1}$. We claim that these tokens are enough to force all vertices in the $Z_2$-game. Since $S$ is a maximum coclique of $G_n$, by Observation \ref{lem:kneser5}, we know that no vertex in $G_{n-1}$ can force any uncoloured vertex (all uncoloured vertices are in $S$). Hence, the $Z_2$-game must be played. Let $H = G_n \left[ V(G_n) \setminus V(G_{n-1}) \right]$ be the graph induced by the vertices in $S$. It is trivial that $H$ is a union of $n-1$ isolated vertices. In the $Z_2$-game, we must provide at least $3$ vertices to the oracle. We consider the following cases depending on the response of the oracle.
		\begin{enumerate}
			\item {\it Case~1. The oracle returns a vertex of $H$}. Then, it is trivial that we can force the returned vertex.
			\item {\it Case~2. The oracle returns two vertices of $H$}. Assume that these vertices are $\{a,n\}$ and $\{b,n\}$. For any $x\not\in \{a,b,n\}$, the vertex $\{a,x\}$ of the Kneser graph $G_{n-1}$ is adjacent to $\{b,n\}$ and not adjacent to $\{a,n\}$. Similarly, $\{b,x\}$ is adjacent to $\{a,n\}$ and not adjacent to $\{b,n\}$. Therefore, we can force these two returned vertices.
			\item {\it Case~3. The oracle returns three vertices of $H$}. Assume that the returned vertices are $\{a,n\},\ \{b,n\}$ and $\{c,n\}$. To prove this, we need to show that for any vertex $V\in \left\{ \{a,n\},\ \{b,n\},\ \{c,n\} \right\}$, we can find a vertex in $G_{n-1}$ that is adjacent to only $V$ and not the other returned vertices. It is not hard to see that for any vertex $V\in \left\{ \{a,n\},\ \{b,n\},\ \{c,n\} \right\}$, there exists a unique vertex in $W\in \{\{a,b\},\ \{a,c\},\ \{b,c\} \}$ such that $W$ is only adjacent to $V$ and not the other two returned vertices. Therefore, we can always force the three returned vertices.
		\end{enumerate}
	
		In each round of the $Z_2$-game on $G_n$, we hand at most three components (which are isolated vertices). Then, no matter how many components are returned by the oracle, we can always force these returned vertices, according the above cases. Consequently, $Z_2(G_n) \leq \binom{n-1}{2}$.
		
	\end{proof}

	\begin{obs}
		For any $n\geq 6$, we have $Z_{n-1}(G_n) = Z(G_n)$.
	\end{obs}
	\begin{proof}
		This is a consequence of the celebrated Erd\"os-Ko-Rado (EKR) theorem (see \cite{GM}). First, note that we need to provide $n$ components to play the $Z_{n-1}$-game. By Lemma~\ref{lem:kneser1}, such components must be a coclique of $K(n,2)$. By the EKR theorem, this is however impossible. Therefore, the optimal strategy is to use this zero-forcing set.
	\end{proof}
	\begin{cor}
		For any $q\geq n-1$, we have $Z_q(G_n) = Z(G_n) = \binom{n}{2}-6$.
	\end{cor}

\section*{Acknowledgments}
The work in this paper was a joint project of the
2022 Discrete Mathematics Research Group at the University of Regina, attended by
all of the authors. 
Dr.~Fallat's research was supported in part by NSERC Discovery Research Grant, Application No.: RGPIN-2019-03934. Dr.~Meagher's research was supported in part by an NSERC Discovery Research Grant, Application No.: RGPIN-03952-2018. 


%
%
\end{document}